\documentclass[12pt]{article}
\usepackage{enumerate,amsfonts,dsfont,amsmath,accents,amsthm,mathrsfs,cases,bbm,amssymb,color}
\usepackage{indentfirst} %首段首行缩进
\usepackage[bookmarks]{hyperref}      %书签
\usepackage{hypernat}
\usepackage{fancyhdr}

\newcounter{RomanNumber}
\newcommand{\MyRoman}[1]{\setcounter{RomanNumber}{#1}\Roman{RomanNumber}}

\allowdisplaybreaks      %多行公式自动分页
\newtheorem{theorem}{Theorem}[section]
\newtheorem{lemma}[theorem]{Lemma}

\theoremstyle{definition}

\newtheorem{remark}[theorem]{Remark}
\numberwithin{equation}{section}

\newcommand{\adj}{\mathrm{adj\,}}
\def\geq{\geqslant}
\def\leq{\leqslant}
\numberwithin{equation}{section}
\begin{document}
\title{Optimal function spaces for the weak continuity of the distributional $k$-Hessian}
\author{Qiang Tu\and Wenyi Chen\and
School of Mathematics and Statistics, Wuhan University
Wuhan 430072, China\footnote{\emph{Email addresses}:~qiangtu@whu.edu.cn(Qiang Tu), wychencn@whu.edu.cn(Wenyi Chen).}.
}
\date{}   %不显示日期
\maketitle
\noindent{\bf Abstract:}  In this paper we introduce the notion of distributional $k$-Hessian associated with Besov type functions in Euclidean $n$-space, $k=2,\ldots,n$. Particularly, inspired by recent work of Baer and Jerison on distributional Hessian determinant, we show that the distributional $k$-Hessian is weak continuous on the Besov space $B(2-\frac{2}{k},k)$, and the result is optimal in the framework of the space $B(s,p)$, i.e., the distributional $k$-Hessian is well defined in $B(s,p)$ if and only if $B(s,p)\subset B_{loc}(2-\frac{2}{k},k)$.
\medskip

\medskip

\noindent{\bf 2010 MR Subject Classification:} 46E35, 46F10, 42B35.
%\medskip

\section{Introduction and main results}
For $k=1,\ldots,n$ and $u\in C_{c}^2(\mathbb{R}^n)$, the $k$-Hessian operator $F_k$ is defined by
$$F_k[u]=S_k(\lambda(D^2 u)),$$
where $\lambda=(\lambda_1,\ldots,\lambda_n)$ denotes the eigenvalues of the Hessian matrix of second derivatives $D^2u$, and $S_k$ is the $k$-th elementary symmetric function on $\mathbb{R}^n$, given by
$$S_k(\lambda)=\sum_{i_1<\cdots<i_k}\lambda_{i_1}\cdots\lambda_{i_k}.$$
Alternatively we may write
$$F_k[u]=[D^2u]_k,$$
where $[A]_k$ denotes the sum of the $k\times k$-principal minors of an $n\times n$ matrix $A$, which may also be called the $k$-trace of $A$. It is well known that the $k$-Hessian is the Laplace operator when $k=1$ and the Monge-Amp\`{e}re  operator when $k=n$.

This paper is devoted to the study of the $k$-Hessian of a nonsmooth map $u$ from $\mathbb{R}^n$ into $\mathbb{R}$, with $2\leq k\leq n$. Starting with the seminal work of Trudinger and  Wang (see \cite{TW1,TW2,TW3,WANG}), it has been known that the $k$-Hessian makes sense as a Radon measure and enjoys the weak continuity property for $k$-admissible functions.
In \cite{Fu1,Fu2}, Fu introduced the space of Monge-Amp\`{e}re functions for which all minors of the Hessian matrices, including in particular the Hessian determinant, are well defined as signed Radon measures and weakly continuous in a certain natural sense. Jerrard \cite{J,J2} extended the notion of Monge-Amp\`{e}re functions and showed analogous continuous property and other structural properties.
Moreover other generalized notion of the $k$-Hessian measure are considered in \cite{AD,AP,DTW}.
Our purpose in this thesis is to extend the definition of the $F_k$ to corresponding classes of functions so that the $k$-Hessian $\mathcal{F}_k[u]$ is a distribution on $\mathbb{R}^n$.
%Now the goal of this thesis is to identify  when the $k$-Hessian $\mathcal{F}_k[u]$ makes sense as a distribution on $\mathbb{R}^n$.
In the case $k=2$, inspired by the results of \cite{IT} characterizing the Hessian determinant on the space $W^{1,2}(\mathbb{R}^2)$, the $2$-Hessian is well defined and continuous on $W^{1,2}(\mathbb{R}^n)$. More precisely, the $2$-Hessian $\mathcal{F}_2[u]$ is defined for all $u\in W^{1,2}(\mathbb{R}^n)$ by
\begin{equation}\label{def01}
\langle \mathcal{F}_2[u], \varphi\rangle:=\sum_{i=1}^n\sum_{j\neq i} \int_{\mathbb{R}^n} \partial_iu\partial_ju\partial_{i,j}\varphi-\frac{1}{2}\partial_iu\partial_iu\partial_{j,j}\varphi-
\frac{1}{2}\partial_ju\partial_ju\partial_{i,i}\varphi dx
\end{equation}
for any $\varphi\in C_c^{\infty}(\mathbb{R}^n)$, where $\partial_i:=\frac{\partial}{\partial x_i}$.
It is obvious to show the weak continuous results by H\"{o}lder inequality.

In the case $3\leq k \leq n$, we consider the $k$-Hessian operator on  a class of Besov spaces on $\mathbb{R}^n$, denote by $B(s,p)=B^{p,p}_s$. In particular, we will show that the $k$-Hessian $\mathcal{F}_k[u]$ is well defined and continuous from the Besov space $B(2-\frac{2}{k},k)$, into the space of distribution.
Moreover, the definition and continuity property is optimal in the framework of the space of $B(s,p)$: the $k$-Hessian operator is continuous on any  $B(s,p)$ satisfying $B(s,p)\subset B_{loc}(2-\frac{2}{k},k)$ and is not  continuous on any other space in the framework of Besov type space.

The initial motivation of our work is the following: Baer and Jerison \cite{BJ} showed that the Hessian determinant operator
$u \mapsto \det(D^2 u): C_c^2(\mathbb{R}^n)\rightarrow \mathcal{D}'(\mathbb{R}^n)$ admits a unique continuous extension, which they denote by $\mathcal{H}$, from the Besov space $B(2-\frac{2}{n},n)$ to the space of distributions
$\mathcal{D}'(\mathbb{R}^n)$,
and  the continuity property fails for any space in the framework of Besov space for which the inclusion $B(s,p) \subset B_{loc}(2-\frac{2}{n},n)$ dose not hold.

We recall that for $1<s<2$ and $1\leq p<\infty$, the Besov space $B(s,p)$ is defined by
$$B(s,p):=\left\{u\in W^{1,p}(\mathbb{R}^n)\mid \left(\int_{\mathbb{R}^n} \int_{\mathbb{R}^n} \frac{|Du(x)-Du(y)|^p}{|x-y|^{n+(s-1)p}} dxdy\right)^{\frac{1}{p}} <\infty\right\},$$
and the norm
$$\|u\|_{s,p}:=\|u\|_{W^{1,p}}+ \left(\int_{\mathbb{R}^n} \int_{\mathbb{R}^n} \frac{|Du(x)-Du(y)|^p}{|x-y|^{n+(s-1)p}} dxdy\right)^{\frac{1}{p}}.$$
Then our first result is the following.
\begin{theorem}\label{thm1}
For $2\leq k\leq n$,
the $k$-Hessian operator $u \mapsto F_k[u]: C_c^2(\mathbb{R}^n) \rightarrow \mathcal{D}'(\mathbb{R}^n)$ can be extended uniquely as a continuous mapping $u \mapsto \mathcal{F}_k[u]: B(2-\frac{2}{k},k)\rightarrow \mathcal{D}'(\mathbb{R}^n)$. Moreover, for all $u_1,u_2\in B(2-\frac{2}{k},k)$ and $\varphi \in C_c^2(\mathbb{R}^n)$, we have
$$|\langle \mathcal{F}_k[u_1]-\mathcal{F}_k[u_2], \varphi \rangle| \leq C \|u_1-u_2\|_{2-\frac{2}{k},k}\left(\|u_1\|^{k-1}_{2-\frac{2}{k},k}+\|u_2\|^{k-1}_{2-\frac{2}{k},k}\right) \|D^2\varphi\|_{L^{\infty}}.$$
\end{theorem}
In the case $k=2$, the results of Theorem \ref{thm1} can be easily deduced by (\ref{def01}),
in which case the regularity index becomes integer and the Besov function space is the usual Sobolev space $W^{1,2}$.
In the case $k=n$, the $k$-Hessian operator  in fact is the  Hessian determinant operator, i.e. $\mathcal{F}_n=\mathcal{H}$, and the analogous  results were already established in \cite{BJ}.
Moreover, in analogy with \cite{BJ}, Theorem \ref{thm1} immediately gives several consequences: in particular, the $k$-Hessian as a distribution is continuous in spaces $W^{1,p}(\mathbb{R}^n)\cap W^{2,q}(\mathbb{R}^n)$ with $1<p,q<\infty$, $\frac{2}{p} +\frac{k-2}{q}=1$ and $k\geq 3$.

 Now we  turn to the optimality result. %  which  asserts that Theorem \ref{thm1} is optimal in the framework of the Besov spaces $B(s,p)$.
 More precisely, the distributional $k$-Hessian is well defined in $B(s,p)$ if and only if $B(s,p)\subset B_{loc}(2-\frac{2}{k}, k)$.
 \begin{theorem}\label{thm2}
 Let $3\leq k\leq n$, $1<p<\infty$ and $1<s<2$ be such that $B(s,p)\nsubseteq B_{loc}(2-\frac{2}{k},k)$. Then there exist a sequence
 $\{u_m\}\subset C_c^{\infty}(\mathbb{R}^n)$ and a function $\varphi \in C_c^{\infty}(\mathbb{R}^n)$ such that
 \begin{equation}\label{formulaorthm31}
 \lim_{m\rightarrow \infty} \|u_m\|_{s,p}=0
 \end{equation}
 and
  \begin{equation}\label{formulaorthm32}
  \lim_{m\rightarrow \infty} \int F_k[u_m] \varphi dx=\infty.
 \end{equation}
 \end{theorem}

 \begin{remark}
 We recall the embedding properties of the Besov  spaces $B(s,p)$ ($1<s<2, 1<p<\infty$) into the space $B_{loc}(2-\frac{2}{k},k)$,
 more details see \cite{ST3} or \cite[page 196]{TH}:
 \begin{enumerate}
\item[(\romannumeral1)]  $s+\frac{2}{k}>2+\max\{0,\frac{n}{p}-\frac{n}{k}\}$, the embedding $B(s,p)\subset B_{loc}(2-\frac{2}{k},k)$ holds;
\item [(\romannumeral2)]   $s+\frac{2}{k}<2+\max\{0,\frac{n}{p}-\frac{n}{k}\}$, the embedding fails;
\item [(\romannumeral3)] $s+\frac{2}{k}=2+\max\{0,\frac{n}{p}-\frac{n}{k}\}$, there are two sub-cases:
 \begin{enumerate}
 \item[(a)] if $p\leq k$, then the embedding $B(s,p)\subset B_{loc}(2-\frac{2}{k},k)$ holds;
 \item[(b)] if $p>k$, the embedding fails.
  \end{enumerate}
\end{enumerate}
 \end{remark}

 In order to prove Theorem \ref{thm2}, we just consider three case:
  \begin{enumerate}
\item[ \MyRoman{1}:] $1<p\leq k$ and $s+\frac{2}{k}<2+\frac{n}{p}-\frac{n}{k}$;
\item [ \MyRoman{2}:]   $k<p$ and $0<s<2-\frac{2}{k}$;
\item [ \MyRoman{3}:] $k<p$ and  $s=2-2/k$.
 \end{enumerate}

 This paper is organized as follows. Some notion about determinant and the proof of Theorem \ref{thm1} are given in Section 2. In Section 3 we show Theorem \ref{thm2} in the case \MyRoman{1}: $1<p\leq k$ and $s+\frac{2}{k}<2+\frac{n}{p}-\frac{n}{k}$. Then we prove Theorem \ref{thm2} in the case \MyRoman{2}: $k<p$ and $0<s<2-\frac{2}{k}$ in Section 4. Finally in Section 5 we establish Theorem \ref{thm2} in the remaining case: $k<p$ and  $s=2-2/k$.

\section{Preliminaries and the proof of Theorem \ref{thm1} }
In this section we prove the continuity results for the $k$-Hessian operator on spaces of Besov type into the space of distributions on $\mathbb{R}^n$. First we recall some notation and facts about determinant.

For integers  $n \geq 2$, we shall use the standard notation for ordered multi-indices
$$I(k,n):=\{\alpha=(\alpha_1,\ldots,\alpha_k) \mid \alpha_i  ~\mbox{integers}, 1\leq \alpha_1 <\cdots< \alpha_k\leq n\}.$$
 Set $I(0,n)=\{0\}$ and $|\alpha|=k$ if $\alpha \in I(k,n)$. If $\alpha\in I(k,n)$, $k=0,1,\ldots,n$,  $\overline{\alpha}$ is the element in $I(n-k,n)$ which complements $\alpha$  in $\{1,2,\ldots,n\}$ in the natural increasing order. So $\overline{0}=\{1,2,\ldots,n\}$.

Given $\alpha=(\alpha_1,\cdot\cdot\cdot,\alpha_k)\in I(k,n)$, we say $i\in \alpha$ if $i$ is one of the indexes $\alpha_1,\cdot\cdot\cdot,\alpha_k$.
 For $i \in \alpha$, $\alpha-i$ means the multi-index of length $k-1$ obtained by removing $i$ from $\alpha$. Similarly for $j\notin \alpha$, $\alpha+j$ means the multi-index of length $k+1$ obtained by reordering naturally the multi-index $(\alpha_1,\ldots,\alpha_k, j)$.

Let $A=(a_{ij})_{n \times n}$ and $B=(b_{ij})_{n \times n}$ be  $n \times n$ matrixes.
Given two ordered multi-indices with $\alpha,\beta\in I (k,n) $, then
$A_{\alpha}^{\beta}$ denotes
the $k \times k $-submatrix of $A$ with rows $(\alpha_1,\ldots,\alpha_k)$ and columns $(\beta_1,\ldots,\beta_k)$. Its determinant will be denoted by
$$M_{\alpha}^{\beta}(A):=\det A_{\alpha}^{\beta}.$$
We denote $\sigma(\alpha,\beta)$ by  the sign of the permutation which reorders $(\alpha,\beta)$ in the natural increasing order and $\sigma(\overline{0},0):=1$.
The adjoint of $A_{\alpha}^{\beta}$ is  defined  by the formula
$$(\adj A_{\alpha}^{\beta})_j^i:= \sigma(i,\beta-i) \sigma(j,\alpha-j) \det A_{\alpha-j}^{\beta-i}~~~~ i \in \beta, j\in \alpha.$$
So Laplace  formulas can be written as
$$M_{\alpha}^{\beta}(A)= \sum_{j \in \alpha} a_{ij} (\adj A_{\alpha}^{\beta})_j^i.$$
And the Binet formulas can be written as (see \cite[page 313]{GMS1})
\begin{equation}\label{binet}
M_{\alpha}^{\beta}(A+B)=\sum_{\alpha'+\alpha^{''}=\alpha;\beta'+\beta^{''}=\beta;|\alpha'|=|\beta'|}
\sigma(\alpha',\alpha^{''})\sigma(\beta',\beta^{''}) M_{\alpha'}^{\beta'}(A) M_{\alpha^{''}}^{\beta^{''}}(B).
\end{equation}
Let $n\geq 2$ and $F:\mathbb{R}^n \rightarrow \mathbb{R}$ be given as
$$F(x)= \prod_{i=1}^n f_i(x_i), ~~~~x=(x_1,\ldots,x_n)\in \mathbb{R}^n,$$
where the function $f_i: \mathbb{R}\rightarrow \mathbb{R}$ for $i=1,\ldots, n$.
 For any $\alpha \in I(k,n)$, it will be  convenient  to introduce the notation
$$F_{\alpha}(x_{\alpha}):= \prod_{i\in \alpha} f_i(x_i),~~~~x_{\alpha}:=(x_{\alpha_1},\ldots,x_{\alpha_k})\in \mathbb{R}^k,$$
$$F_{\overline{\alpha}}(x_{\overline{\alpha}}):= \prod_{i\in \overline{\alpha}} f_i(x_i) ,~~~~x_{\overline{\alpha}}:=(x_{\overline{\alpha}_1},\ldots,x_{\overline{\alpha}_{n-k}})\in \mathbb{R}^{n-k}.$$

We now turn to the proof of Theorem \ref{thm1}, which actually can be seen as an immediate consequence following from the standard approximation argument
if we have proven the following result.
\begin{theorem}\label{thm21}
Let $3 \leq k\leq n$. Then for all $u_1,u_2,\varphi \in C_c^2(\mathbb{R}^n)$, we have
\begin{equation}\label{formularorthm1}
\left|\int_{\mathbb{R}^n} (F_k[u_1]-F_k[u_2]) \varphi dx\right| \leq C \|u_1-u_2\|_{2-\frac{2}{k},k}\left(\|u_1\|^{k-1}_{2-\frac{2}{k},k}+\|u_2\|^{k-1}_{2-\frac{2}{k},k}\right) \|D^2\varphi\|_{L^{\infty}}.
\end{equation}
\end{theorem}
In order to prove the above theorem, we need the following  extension result which is inspired from the
work of Baer-Jerison \cite{BJ}.
\begin{lemma}\label{lem21}
Let $3\leq k\leq n$, $\alpha \in I(k,n)$ and $u, \varphi \in C_c^2(\mathbb{R}^n)$. Then
\begin{equation}
\int_{\mathbb{R}^n} M_{\alpha}^{\alpha} (D^2 u) \varphi dx=\sum_{i\in \alpha+(n+1)} \sum_{j\in \alpha+(n+2) } \int_{\mathbb{R}^n\times(0,1)^2} \adj\left((D^2 U)_{\alpha+(n+2)}^{\alpha+(n+1)}\right)_j^i     \partial_{i,j} \Phi d\widetilde{x}.
\end{equation}
for any extensions $U$ and $\Phi \in C^2_c(\mathbb{R}^n\times[0,1)\times[0,1))$ of $u$ and $\varphi$, respectively, here $\widetilde{x}=(x,x_{n+1},x_{n+2})$.
\end{lemma}
\begin{proof}
Denote $V:=U|_{x_{n+2}=0}$, $\Psi:=\Phi|_{x_{n+2}=0}$ and $\partial_i:=\frac{\partial}{\partial x_i}$. Then
\begin{align*}
\int_{\mathbb{R}^n} M_{\alpha}^{\alpha} (D^2 u) \varphi dx
&=-\int_{\mathbb{R}^{n}\times(0,1)}\partial_{n+1}\left(M_{\alpha}^{\alpha} (D^2 V) \Psi \right)dx dx_{n+1}\\
&=-\int_{\mathbb{R}^{n}\times(0,1)}\partial_{n+1}\left(M_{\alpha}^{\alpha} (D^2 V) \right)\Psi dx dx_{n+1}-\int_{\mathbb{R}^{n}\times(0,1)}M_{\alpha}^{\alpha} (D^2 V) \partial_{n+1}\Psi dx dx_{n+1}.
\end{align*}
We denote the first part integral on the right-hand side by  \MyRoman{1}, using Laplace  formulas we obtain
\begin{align*}
\MyRoman{1}&=-\sum_{i\in\alpha}\sum_{j\in \alpha} \int_{\mathbb{R}^{n}\times(0,1)} \sigma (i,\alpha-i) \sigma(j,\alpha-j) \partial_{n+1}\partial_i \partial_j V M_{\alpha-j}^{\alpha-i} (D^2 V) \Psi dx dx_{n+1}\\
&=\sum_{i\in\alpha} \sum_{j\in \alpha} \sigma (i,\alpha-i) \sigma(j,\alpha-j)\int_{\mathbb{R}^{n}\times(0,1)} \partial_{n+1} \partial_j V \left(\partial_i (M_{\alpha-j}^{\alpha-i} (D^2 V)) \Psi +  M_{\alpha-j}^{\alpha-i} (D^2 V) \partial_i\Psi\right) dx dx_{n+1}.
\end{align*}
Since
$$\sum_{i\in \alpha} \partial_i \left((\adj(D^2V))_{\alpha}^{\alpha}\right)_j^i=0$$
for any $j\in \alpha$, it follows that
\begin{align*}
\MyRoman{1}&=\sum_{i\in\alpha} \sum_{j\in \alpha} \int_{\mathbb{R}^{n}\times(0,1)} \sigma (i,\alpha-i) \sigma(j,\alpha-j) \partial_{n+1} \partial_j V M_{\alpha-j}^{\alpha-i} (D^2 V) \partial_i\Psi dx dx_{n+1}\\
&=\sum_{i\in\alpha}  \int_{\mathbb{R}^{n}\times(0,1)} \sigma (i,\alpha-i) \sigma(n+1,\alpha-i)  M_{\alpha}^{\alpha+(n+1)-i} (D^2 V) \partial_i\Psi dx dx_{n+1}\\
&=\sum_{i\in\alpha} -\sigma (\alpha+(n+1)-i,i) \int_{\mathbb{R}^{n}\times(0,1)} M_{\alpha}^{\alpha+(n+1)-i} (D^2 V) \partial_i\Psi dx dx_{n+1}.
\end{align*}
Hence
$$
\int_{\mathbb{R}^n} M_{\alpha}^{\alpha} (D^2 u) \varphi dx=\sum_{i\in \alpha+(n+1)} -\sigma(\alpha+(n+1)-i,i) \int_{\mathbb{R}^n\times(0,1)} \left(M_{\alpha}^{\alpha-i+(n+1)}(D^2 U)\partial_i \Phi\right) |_{x_{n+2}=0} dx dx_{n+1}.
$$
It is well known consequence of integration by parts that  the right-hand side of the above identity can be written as
\begin{equation}\label{formulathm102}
\sum_{i\in \alpha+(n+1)} \sigma(\alpha+(n+1)-i,i) A(i),
\end{equation}
where
\begin{align*}
A(i):&=\int_{\mathbb{R}^n\times (0,1)^2} \partial_{n+2} \left(M_{\alpha}^{\alpha-i+(n+1)}(D^2 U)\partial_i \Phi \right) d\widetilde{x}\\
&=\int_{\mathbb{R}^n\times (0,1)^2} \partial_{n+2} \left(M_{\alpha}^{\alpha-i+(n+1)}(D^2 U)\right)\partial_i \Phi d\widetilde{x}+\int_{\mathbb{R}^n\times (0,1)^2} M_{\alpha}^{\alpha-i+(n+1)}(D^2 U)\partial_{i,n+2} \Phi d\widetilde{x}.
\end{align*}
 For simplicity we may set $\beta:= \alpha-i+(n+1)$. Obviously,
\begin{equation}\label{formulathm103}
\partial_{n+2} \left(M_{\alpha}^{\beta}(D^2 U)\right)=\sum_{j\in\alpha}\sum_{t\in \beta} \sigma(j,\alpha-j) \sigma(t,\beta-t) \partial_{n+2} \partial_j \partial_t U M_{\alpha-j}^{\beta-t}(D^2 U),
\end{equation}
and for any $j\in \alpha$,
\begin{align*}
&\sum_{j\in \alpha}\partial_j \left( \sigma(\alpha-j,j) M_{\alpha-j+(n+2)}^{\beta} (D^2 U)\right)\\
&=\sum_{j\in \alpha} \sigma(\alpha-j,j) \sum_{s\in \alpha-j} \sum_{t\in \beta} \sigma(s,\alpha-j+(n+2)-s)\sigma(t,\beta-t) \partial_j \partial_s \partial_t U M_{\alpha-j+(n+2)-s}^{\beta-t}(D^2 U)\\
&+\sum_{j\in \alpha} \sigma(\alpha-j,j) \sum_{t\in \beta} \sigma((n+2),\alpha-j)\sigma(t,\beta-t) \partial_j \partial_{n+2} \partial_t U M_{\alpha-j}^{\beta-t}(D^2 U)\\
&=\sum_{j\in \alpha} \sum_{s\in \alpha-j}\sigma(\alpha-j,j) \sigma(s,\alpha-j-s)\sum_{t\in \beta} \sigma(t,\beta-t) \partial_j \partial_s \partial_t U M_{\alpha-j+(n+2)-s}^{\beta-t}(D^2 U)\\
&+\sum_{j\in \alpha} \sum_{t\in \beta} \sigma(j,\alpha-j) \sigma(t,\beta-t) \partial_j \partial_{n+2} \partial_t U M_{\alpha-j}^{\beta-t}(D^2 U).
\end{align*}
Note that for any $i_1,i_2\in \alpha$ with $i_1\neq i_2$
\begin{align*}
\sigma(\alpha-i_1,i_1) \sigma(i_2, \alpha-i_1-i_2)=(-1)^{k-1} \sigma(i_1,\alpha-i_1-i_2) \sigma(i_2, \alpha-i_1-i_2) (-1)^{\tau(i_1,i_2)},
\end{align*}
where
$$\tau(i_1,i_2):=\begin{cases}
1,~~~~i_1>i_2,\\
0,~~~~i_1<i_2,
\end{cases}
$$
which implies that
$$\sigma(\alpha-i_1,i_1) \sigma(i_2, \alpha-i_1-i_2)=-\sigma(\alpha-i_2,i_2) \sigma(i_1, \alpha-i_1-i_2).$$
Combing with the above results, we can easily obtain
\begin{equation}
\partial_{n+2} \left(M_{\alpha}^{\beta}(D^2 U)\right)=\sum_{j\in \alpha}\partial_j \left( \sigma(\alpha-j,j) M_{\alpha-j+(n+2)}^{\beta} (D^2 U)\right).
\end{equation}
Then taking the sum in $i$ and recalling (\ref{formulathm102}), we have
\begin{align*}
\int_{\mathbb{R}^n} M_{\alpha}^{\alpha} (D^2 u) \varphi dx&=\int_{\mathbb{R}^n\times(0,1)^2}\sum_{i\in \alpha+(n+1)} \sigma(\alpha+(n+1)-i,i) \\
&\left\{-\sum_{j\in\alpha} \sigma(\alpha-j,j) M_{\alpha-j+(n+2)}^{\beta} (D^2 U) \partial_{i,j} \Phi
+ M_{\alpha}^{\beta}(D^2 U) \partial_{i,n+2} \Phi \right\} d\widetilde{x},
\end{align*}
which completes the proof.
\end{proof}

\begin{proof}[\bf Proof of Theorem \ref{thm21}]
According to  a well known extension theorem  of Stein in \cite{ST1,ST2}, there is a bounded linear extension operator
$$E: B(2-\frac{2}{k},k) \rightarrow W^{2,k} (\mathbb{R}^n\times [0,1)^2).$$
Let $U_1, U_2\in C_c^2(\mathbb{R}^n)$  be extensions of $u_1$ and $u_2$ to $\mathbb{R}^n\times(0,1)^2$, respectively, such that
$$\|D^2 U_i\|_{L^k(\mathbb{R}^n\times(0,1)^2)}\leq C\|u_i\|_{2-\frac{2}{k},k},~~i=1,2,$$
and
$$\|D^2U_1-D^2U_2\|_{L^k(\mathbb{R}^n\times(0,1)^2)}\leq C\|u_1-u_2\|_{2-\frac{2}{k},k}.$$
Let $\Phi\in C_c^2(\mathbb{R}^n\times[0,1)^2)$ be an extension of $\varphi$ such that
$$\|D^2 \Phi\|_{L^{\infty}(\mathbb{R}^n\times(0,1)^2)} \leq C \|D^2 \varphi\|_{L^{\infty}(\mathbb{R}^n)}.$$
Since
$$|M_{\alpha}^{\beta} (A)-M_{\alpha}^{\beta} (B)|\leq C\left(|A|+|B|\right)^{k-1}|A-B|$$
for any $\alpha,\beta \in I(k,n+2)$ and $(n+2)\times(n+2)$ matrixes $A,B$. It follows from Lemma \ref{lem21} and H\"{o}lder's inequality that
\begin{align*}
&\left|\int_{\mathbb{R}^n} (F_k[u_1]-F_k[u_2]) \varphi dx\right| \\
&\leq \sum_{\alpha\in I(k,n)} \sum_{i\in \alpha+(n+1)} \sum_{j\in \alpha+(n+2)}\int_{\mathbb{R}^n\times (0,1)^2}|M_{\alpha-j+(n+2)}^{\alpha-i+(n+1)} (D^2 U_1)-M_{\alpha-j+(n+2)}^{\alpha-i+(n+1)} (D^2 U_2)|
|\partial_{i,j}\Phi| d\widetilde{x} \\
&\leq C \int_{\mathbb{R}^n\times (0,1)^2}(|D^2U_1|+|D^2 U_2|)^{k-1}|D^2(U_1-U_2)||D^2 \Phi| d\widetilde{x} \\
&\leq C \|u_1-u_2\|_{2-\frac{2}{k},k}\left(\|u_1\|^{k-1}_{2-\frac{2}{k},k}+\|u_2\|^{k-1}_{2-\frac{2}{k},k}\right) \|D^2\varphi\|_{L^{\infty}}.
\end{align*}
This completes the proof of Theorem \ref{thm21}.
\end{proof}

\section{Optimality results \MyRoman{1}: $1<p\leq k,$ $s+\frac{2}{k}<2+\frac{n}{p}-\frac{n}{k}$}
In this section we establish the optimality result of Theorem 1.2 in the case $1<p\leq k$ and  $s+\frac{2}{k}<2+\frac{n}{p}-\frac{n}{k}$.
For this, we need the following lemma

\begin{lemma}\label{lem11}
Let $g\in C_c^{\infty}(B(0,1))$ be given as
\begin{equation}\label{formulaorlem01}
g(x)=\int_0^{|x|} h(r) dr
\end{equation}
for any $x\in \mathbb{R}^n$,  where $h\in C_c^{\infty}((0,1))$ and satisfies
$$\int_0^1 h(r)dr=0,~~~~\int_0^1h^{k}(r)r^{-k+n+1} dr\neq 0.$$
Then
\begin{equation}\label{formulaorlem02}
\sum_{\alpha \in I(k,n)}\int_{B(0,1)} M_{\alpha}^{\alpha}( D^2 g(x)) |x|^2 dx\neq 0.
\end{equation}
\end{lemma}
\begin{proof}
According to the symmetry of integral, it is sufficient to show
\begin{equation}\label{formulaorlem03}
\int_{B(0,1)} M_{\alpha}^{\alpha}( D^2 g(x)) |x|^2 dx\neq 0
\end{equation}
for any $\alpha \in I(k,n)$.  It is easy to see that
$$D^2 g=\frac{1}{|x|^3}(A+B),$$
where $A=(a_{ij})_{n\times n}$ and $B=(b_{ij})_{n\times n}$ are $n\times n$ matrices such that
$$a_{i,j}=h'(|x|)|x|x_ix_j,~~b_{ij}=h(|x|)(\delta_{i}^{j}|x|^2-x_ix_j),~~~~i,j=1,\ldots,n.$$
Using Binet formula and the fact $\mbox{rank}(A)=1$, one has
\begin{align*}
M_{\alpha}^{\alpha}(A+B)&=M_{\alpha}^{\alpha}(B)+\sum_{i\in \alpha}\sum_{j\in\alpha}\sigma(i,\alpha-i)\sigma(j,\alpha-j)a_{ij}M_{\alpha-i}^{\alpha-j}(B)\\
&=h^k(|x|)|x|^{2k-2}(|x|^2-\sum_{i\in\alpha} x_i^2)+h'(|x|)h^{k-1}(|x|)|x|\cdot \MyRoman{1},
\end{align*}
where
\begin{align*}
\MyRoman{1}:&=\sum_{i\in \alpha}\sum_{j\in\alpha}\sigma(i,\alpha-i)\sigma(j,\alpha-j) x_ix_j M_{\alpha-i}^{\alpha-j}\left((|x|^2\delta_i^j-x_ix_j)_{n\times n}\right)\\
&=\sum_{i\in \alpha} x_i^2 M_{\alpha-i}^{\alpha-i}\left((|x|^2\delta_i^j-x_ix_j)_{n\times n}\right)+\sum_{i\in \alpha} \sum_{j\in \alpha-i} \sigma(i,\alpha-i)\sigma(j,\alpha-j) x_ix_j M_{\alpha-i}^{\alpha-j}\left((|x|^2\delta_i^j-x_ix_j)_{n\times n}\right)\\
&=\sum_{i\in \alpha} x_i^2 |x|^{2k-4}(|x|^2-\sum_{j\in \alpha-i}x_j^2)+\sum_{i\in \alpha} \sum_{j\in \alpha-i}
x_ix_j x_ix_j |x|^{2k-4}\\
&=|x|^{2k-2}\sum_{i\in\alpha} x_i^2.
\end{align*}
Hence
\begin{align*}
\int_{B(0,1)}M_{\alpha}^{\alpha}( D^2 g)|x|^2 dx=\int_{B(0,1)} |x|^{-3k+2}M_{\alpha}^{\alpha}(A+B) dx=\MyRoman{2}-\MyRoman{3}+\MyRoman{4},
\end{align*}
where
$$\MyRoman{2}:=\int_{B(0,1)} h^k(|x|) |x|^{-k+2} dx,$$
$$\MyRoman{3}:=\int_{B(0,1)} h^k(|x|) |x|^{-k} \sum_{i\in \alpha} x_i^2dx,$$
and
$$\MyRoman{4}:=\int_{B(0,1)} h^{k-1}(|x|)h'(|x|)|x|^{-k+1} \sum_{i\in \alpha} x_i^2dx.$$
Then integration in polar coordinates gives
%\begin{align*}
%\MyRoman{4}&=\int_0^1\int_0^{\pi}\cdot\cdot\cdot\int_0^{\pi}\int_0^{2\pi}h^{k-1}(r)h'(r)r^{-k+n+2} \\
%&\left(\sin^n\theta_1 \sin^{n-1}\theta_{2}\cdot\cdot\cdot\sin^{k+1}\theta_{n-k}
%\cdot\sin^{k-2}\theta_{n-k+1}\cdot\cdot\cdot\sin\theta_{n-2} \right) drd\theta_1\cdot\cdot\cdot d\theta_{n-1}\\
%&=2\pi I(n)I(n-1)\cdot\cdot\cdot I(k+1)I(k-2)\cdot\cdot\cdot I(1)\int_0^1 h^{k-1}(r)h'(r)r^{-k+n+2}dr\\
%&=\frac{k-n-2}{n} 2\pi\prod_{i=1}^{n-2}I(i) \int_0^1 h^k(r) r^{-k+n+1} dr,
%\end{align*}
$$\MyRoman{4}=\frac{k-n-2}{n} 2\pi\prod_{i=1}^{n-2}I(i) \int_0^1 h^k(r) r^{-k+n+1} dr,$$
where $I(s)=\int_0^{\pi} \sin^s \theta d\theta$. Similarly,
$$\MyRoman{3}=\frac{k}{n} 2\pi\prod_{i=1}^{n-2}I(i) \int_0^1 h^k(r) r^{-k+n+1} dr,$$
and
$$\MyRoman{2}= 2\pi\prod_{i=1}^{n-2}I(i) \int_0^1 h^k(r) r^{-k+n+1} dr,$$
which implies (\ref{formulaorlem03}), and then the proof is complete.
\end{proof}

\begin{theorem}
 Let $3\leq k\leq n$, $1<p\leq k$ and  $0<s<2$ with $s+\frac{2}{k}<2+\frac{n}{p}-\frac{n}{k}$. Then there exist a sequence $\{u_m\}\subset C_c^{\infty}(\mathbb{R}^n)$ and $\varphi \in C_c^{\infty}(\mathbb{R}^n)$ satisfying (\ref{formulaorthm31}) and (\ref{formulaorthm32}).
\end{theorem}
\begin{proof}
Consider $u_{m}: \mathbb{R}^n \rightarrow \mathbb{R}$ defined by
$$u_{m}(x)=m^{-\rho}g(mx),~~~~m>1,$$
where $g$ is given as (\ref{formulaorlem01}) and $\rho$ is a constant such that
\begin{equation}\label{formulaor103}
s-\frac{n}{p}<\rho<2-\frac{n}{k}-\frac{2}{k}.
\end{equation}
On the one hand, we have
$$\|u_{m}\|_{s,p} \leq \|u_{m}\|_{L^p}^{1-\frac{s}{2}} \|D^2 u_{m}\|_{L^p}^{\frac{s}{2}} \leq  m^{s-\rho-\frac{n}{p}} \|g\|_{L^p}^{1-\frac{s}{2}} \|D^2 g\|_{L^p}^{\frac{s}{2}},$$
which implies (\ref{formulaorthm31}).
On the other hand, let $\varphi\in C_c^{\infty}(\mathbb{R}^n)$ be such that $\varphi(x)=|x|^2+O(|x|^3)$ as $x\rightarrow 0$. Then
\begin{align*}
\int_{\mathbb{R}^n} F_k[u_{m}] \varphi dx&=\sum_{\alpha\in I(k,n)}m^{-(\rho-2)k}\int_{\mathbb{R}^n}  M_{\alpha}^{\alpha} (D^2 g(mx)) \varphi(x) dx\\
&=\sum_{\alpha\in I(k,n)}m^{-(\rho-2)k-n}\int_{B(0,1)}  M_{\alpha}^{\alpha} (D^2 g) \varphi(\frac{x}{m}) dx\\
&=m^{2k-\rho k-n-2} \sum_{\alpha\in I(k,n)} \int_{B(0,1)}  M_{\alpha}^{\alpha} (D^2 g) |x|^2 dx+O(m^{2k-\rho k-n-3}).
\end{align*}
Collecting Lemma \ref{lem11} and (\ref{formulaor103}), it follows that
$$\left|\int_{\mathbb{R}^n} F_k[u_{m}] \varphi dx\right| \geq C m^{2k-\rho k-n-2}\rightarrow \infty~~~~\mbox{as}~m\rightarrow \infty.$$
Hence the theorem is proved completely.
\end{proof}

\section{Optimality results \MyRoman{2}: $k<p,$ $s<2-\frac{2}{k}$}
In this section we consider the case $p>k$ and $0<s<\frac{2}{k}$ for Theorem 1.2.
We begin with the following simple lemma which is a formula due to Chen \cite{CB} for the Hessian determinant of functions as tensor product.
\begin{lemma}\label{lemma21}
Let $2\leq k\leq n$, $\alpha\in I(k,n)$ and $F:\mathbb{R}^n \rightarrow \mathbb{R}$ be given by a tensor product
$$F(x)= \prod_{i=1}^n f_i(x_i), ~~~~x=(x_1,\ldots,x_n)\in \mathbb{R}^n,$$
where $f_i\in C^2(\mathbb{R})$, $i=1,\ldots, n$. Then
$$M_{\alpha}^{\alpha}(D^2F)=(F_{\overline{\alpha}}(x_{\overline{\alpha}}))^k(F_{\alpha}(x_{\alpha}))^{k-2} \left\{\left(\prod_{i\in\alpha} g_i(x_i)\right) +\sum_{j\in \alpha} \left(\prod_{i\in\alpha-j} g_i(x_i)\right)  [f'_j(x_j)]^2\right\}$$
with
$$g_i(x_i)=f^{''}_i(x_i)f_i(x_i)-[f'_i(x_i)]^2,~~~~i=1,\ldots,n.$$
\end{lemma}

 Let $\gamma=(1,\ldots,k-1)\in I(k-1,n)$ and $\Omega\subset \{x \in \mathbb{R}^n\mid x_{i}>0~\mbox{for all}~i\in\overline{\gamma}\}$ be a nonempty open set. For any $m\in \mathbb{N}^{+}$, define $u_m: \mathbb{R}^n \rightarrow \mathbb{R}$
\begin{equation}\label{formulaor201}
u_m=m^{-\rho} \chi(x) P_{\gamma}(x_{\gamma})\cdot Q_{\overline{\gamma}}(x_{\overline{\gamma}}),
\end{equation}
where the functions $P,Q: \mathbb{R}^n\rightarrow \mathbb{R}$ are given by
 $$P(x):=\prod_{i=1}^{n}\sin^2(mx_i),~~~~Q(x):=\prod_{i=1}^{n}x_i.$$
Assume that $\max\{s,2-\frac{4}{k}\}<\rho<2-\frac{2}{k}$, and $\chi\in C_c^{\infty}(\mathbb{R}^n)$ is a smooth cutoff function with $\chi=1$ on $\Omega$.

\begin{theorem}
Let $3\leq k\leq n$, $k<p<\infty$ and $0<s<2-\frac{2}{k}$. Let $u_m\in C_c^{\infty}(\mathbb{R}^n)$ and $\Omega$ be defined as above. Then for any $\Omega'\subset\subset \Omega$, $\varphi\in C_c^{\infty}(\Omega)$ with $\varphi\geq 0$ and $\varphi=1$ on $\Omega'$, the functions $u_m$ satisfy (\ref{formulaorthm31}) and (\ref{formulaorthm32}).
\end{theorem}
\begin{proof}
According to the facts that $\|u_m\|_{L^{\infty}} \leq C m^{-\rho}$ and $\|D^2u_m\|_{L^{\infty}}\leq C m^{2-\rho}$, where constants depending on the measure of $\mbox{spt} \chi$, it follows that
$$\|u_m\|_{s,p}\leq C\|u_m\|_{L^p}^{1-\frac{s}{2}}\|u_m\|_{W^{2,p}}^{\frac{s}{2}}\leq Cm^{s-\rho},$$
which implies (\ref{formulaorthm31}).

On the other hand, it follows from our hypotheses on the cutoff function $\chi$ that
$$u_m(x)=m^{-\rho} P_{\gamma}Q_{\overline{\gamma}}~~~~x\in \Omega.$$
For simplicity, we may set
$$I_c:=\{\alpha\in I(k,n)\mid \alpha=(\alpha',\alpha^{''}), \alpha' \subset \gamma,\alpha^{''} \subset \overline{\gamma}, |\alpha'|=c\}.$$
Then
\begin{equation}\label{formulaor204}
F_k[u_m]=\sum_{\alpha\in I(k,n)} M_{\alpha}^{\alpha} (D^2 u_m)=\sum_{c=0}^{k-1} \sum_{\alpha\in I_c} M_{\alpha}^{\alpha} (D^2 u_m).
\end{equation}
Hence
\begin{equation}\label{formulaor205}
\left| \int F_k[u_m] \varphi dx\right|\geq \left|\sum_{\alpha\in I_{k-1}} \int M_{\alpha}^{\alpha} (D^2 u_m)\varphi dx \right| -\sum_{c=0}^{k-2} \sum_{\alpha\in I_c}\left| \int M_{\alpha}^{\alpha} (D^2 u_m) \varphi dx\right|.
\end{equation}
For any $\alpha\in I_{k-1}$, i.e., $\alpha=\gamma+j$ with $j\in \overline{\gamma}$, by using Lemma \ref{lemma21} we obtain that
\begin{align*}
M_{\alpha}^{\alpha} (D^2 u_m) &= m^{-\rho k}(Q_{\overline{\gamma}-j})^k \det \left( D^2 (P_{\gamma}\cdot x_j)\right)\\
&=(-1)^{k-1}2^km^{2k-2-\rho k}(Q_{\overline{\gamma}-j})^k (P_{\gamma})^{k-1}x_j^{k-2}\left(\sum_{i\in \gamma}\cos^2 (m x_i)\right).
\end{align*}
So
\begin{align*}\label{formulaor206}
\left|\sum_{\alpha\in I_{k-1}} \int M_{\alpha}^{\alpha} (D^2 u_m)\varphi dx \right|&=2^km^{2k-2-\rho k}
\left|\sum_{j\in \overline{\gamma}} \int (Q_{\overline{\gamma}-j})^k (P_{\gamma})^{k-1}x_j^{k-2}\left(\sum_{i\in \gamma}\cos^2 (m x_i)\right) \varphi dx\right|\\
&\geq C m^{2k-2-\rho k} \sum_{i\in \gamma}\sum_{j\in \overline{\gamma}} \int_{\Omega'} (Q_{\overline{\gamma}-j})^k (P_{\gamma})^{k-1} x_j^{k-2}\cos^2 (m x_i) dx\\
&\geq C m^{2k-2-\rho k}.
\end{align*}
For any $\alpha \in I_c(0\leq c\leq k-2)$, i.e., $\alpha=(\alpha',\alpha^{''})$ with $\alpha' \subset \gamma,\alpha^{''} \subset \overline{\gamma}$.
Similarly,
$$M_{\alpha}^{\alpha} (D^2 u_m)=(-1)^{k-1}2^cm^{2c-\rho k}(P_{\gamma-\alpha'}Q_{\overline{\gamma}-\alpha^{''}} )^k (P_{\alpha'})^{k-1}(Q_{\alpha^{''}})^{k-2}(k-c-1+2\sum_{i\in \alpha'} \cos^2 (mx_i)).$$
Hence
\begin{align*}
&\sum_{c=0}^{k-2} \sum_{\alpha\in I_c} \left| \int M_{\alpha}^{\alpha} (D^2 u_m) \varphi dx\right|\\
&\leq \sum_{c=0}^{k-2} 2^c m^{2c-\rho k} \sum_{\alpha\in I_c} \left| \int (P_{\gamma-\alpha'}Q_{\overline{\gamma}-\alpha^{''}} )^k (P_{\alpha'})^{k-1}(Q_{\alpha^{''}})^{k-2}(k-c-1+2\sum_{i\in \alpha'} \cos^2 (mx_i))  \varphi dx\right|\\
&\leq C m^{2k-4-\rho k}.
\end{align*}
By the hypothesis $\max\{s,2-\frac{4}{k}\}<\rho <2-\frac{2}{k}$, we may easily show (\ref{formulaorthm32}). This completes the proof of the theorem.
\end{proof}

\section{Optimality results \MyRoman{3}: $k<p,$ $s=2-2/k$}
We conclude the proof of Theorem \ref{thm2} by showing the results in the remaining case $p>k$ and $s=2-\frac{2}{k}$.

\begin{theorem}\label{thm51}
Let $3\leq k\leq n$, $k<p$ and $s=2-2/k$.  Then there exist a sequence
 $\{u_m\}\subset C_c^{\infty}(\mathbb{R}^n)$ and a function $\varphi \in C_c^{\infty}(\mathbb{R}^n)$ satisfying (\ref{formulaorthm31}) and (\ref{formulaorthm32}).
\end{theorem}

For $m\in \mathbb{N}$ with $m\geq 2$, let
$$n_l=m^{k^{3l}},~~~~l=1,2,\ldots, m.$$
Define $g_l: \mathbb{R}^n\rightarrow \mathbb{R}$ as follows
$$g_l(x)=P_{l, \gamma}(x_{\gamma})\cdot Q_{\overline{\gamma}}(x_{\overline{\gamma}}),$$
where
$$P_l:=\prod_{i=1}^{n} \sin^2(n_l x_i),\ \  Q:=\prod_{j=1}^{n} x_{j},\ \ \  \gamma=(1,\cdot\cdot\cdot,k-1)\in I(k-1,n).$$
Then define $u_m : \mathbb{R}^n \rightarrow \mathbb{R}$ by
\begin{equation}\label{formulaor32}
u_m(x)=\chi(x) \sum_{i=1}^m \frac{1}{n_l^{2-\frac{2}{k}}l^{\frac{1}{k}}}g_l(x),
\end{equation}
where $\chi(x) \in C_c^{\infty}(\mathbb{R}^n)$ is a smooth cutoff function satisfying $\chi(x)=1$ for $x\in (0,2\pi)^n$.
In order to end the proof, some results are introduced as follows.

\begin{lemma}
Let $3\leq k\leq n$, $k<p<\infty$ and $u_m$  defined by (\ref{formulaor32}). Then
$$\sup_{m\in \mathbb{N}} \|u_m\|_{2-\frac{2}{k},p}<\infty.$$
\end{lemma}
\begin{proof}
The proof is closely the same as the proof of boundedness (5.3) in \cite{BJ}. According to the standard estimates  for products in Besov space,  it suffices to estimate $\|w_m\|_{2-\frac{2}{k},p}$ on $[0,2\pi]^n$,  where
$$w_m=\sum_{l=1}^m \frac{1}{n_l^{2-\frac{2}{k}}l^{\frac{1}{k}}} P_{l,\gamma}=\sum_{l=1}^m \frac{1}{n_l^{2-\frac{2}{k}}l^{\frac{1}{k}}}  \prod_{i=1}^{k-1} \sin^2(n_l x_i).$$
The Littlewood-Paley characterization of the Besov space $B(2-\frac{2}{k},p)([0,2\pi]^n)$ (see, e.g. \cite{TH}) implies
\begin{equation}\label{formulaor317}
\|w_m\|_{2-\frac{2}{k},p} \leq C\left( \|w_m\|^p_{L^p([0,2\pi]^n)}+ \sum_{j=1}^{\infty} 2^{(2-\frac{2}{k})jp} \|T_j(w_m)\|^p_{L^p([0,2\pi]^n)}\right)^{\frac{1}{p}}.
\end{equation}
Here the operators $T_j:L^p\rightarrow L^p$ are defined by
$$T_j\left(\sum a_l e^{il\cdot x}\right)=\sum_{2^j\leq |l|\leq 2^{j+1}} \left( \rho(\frac{|l|}{2^{j+1}})- \rho(\frac{|l|}{2^{j}})\right)a_le^{il\cdot x},$$
 where $\rho\in C_c^{\infty}(\mathbb{R})$ is a suitably chosen bump function.

However, it is clear that $\|w_m\|_{L^p}$ is uniformly bounded because of the definition of $n_l$,
while an argument similar to the one used to prove the (5.3) in \cite{BJ} shows that
$$ \sum_{j=1}^{\infty} 2^{(2-\frac{2}{k})jp} \|T_j(w_m)\|^p_{L^p([0,2\pi]^n)}\leq C \sum_{l=1}^{\infty}\frac{1}{l^{\frac{p}{k}}},$$
where $C>0$ is a constant only depending on $k$. This gives the desired result.
\end{proof}

\begin{lemma}
Let $3\leq k\leq n$, $k<p<\infty$ and $u_m$  defined by (\ref{formulaor32}). And $\varphi\in C_c^{\infty}(\mathbb{R}^n)$ is defined by
$$\varphi(x)=\prod_{i=1}^n \varphi_i(x_i),~~~~x=(x_1,\ldots,x_n)\in \mathbb{R}^n,$$
where $\varphi_i\geq 0$ and $\mbox{spt}\varphi_i\subset (0, 2\pi)$  for $i=1,\ldots, n$. Then there exist $c>0$ and $K_0$ such that
$$\left|\int_{\mathbb{R}^n} F_k[ u_m] \varphi dx\right|\geq c(\log m)-K_0.$$
\end{lemma}
\begin{proof}
It follows from our hypotheses on the cutoff function $\chi$ that
$$u_m(x)=\sum_{l=1}^m \frac{1}{n_l^{2-\frac{2}{k}}l^{\frac{1}{k}}}g_l(x),~~~~x\in (0,2\pi)^n.$$
Fix $c=0,1,\ldots, k-1$, and set
$$I_c:=\left\{\alpha\in I(k,n)\mid \alpha=(\alpha',\alpha''), \alpha' \subset \gamma,\alpha''\subset \overline{\gamma}, |\alpha'|=c\right\}.$$
Using the definition of $k$-Hessian we have
\begin{align*}\label{formulaor33}
\left|\int_{\mathbb{R}^n} F_k[ u_m] \varphi dx\right|&=\left| \sum _{\alpha\in I(k,n)} \int M_{\alpha}^{\alpha}( D^2 u_m)\varphi dx\right|\\
&\geq \MyRoman{1}-\MyRoman{2},
\end{align*}
where
\begin{equation}\label{formulaor34}
\MyRoman{1}:=\left|\sum_{\alpha\in I_{k-1}} \int M_{\alpha}^{\alpha}( D^2 u_m)\varphi dx\right|
\end{equation}
and
\begin{equation}\label{formulaor35}
\MyRoman{2}:=\sum_{s=0}^{k-2}\left| \sum_{\alpha\in I_s} \int M_{\alpha}^{\alpha}(D^2u_m)\varphi dx\right|.
\end{equation}
We shall divide the proof into seven steps.

Step 1: Estimate of \MyRoman{1}.

For any $\alpha \in I_{k-1}$, we can write $\alpha=\gamma+j=(1,2,\ldots,k-1,j)$ with $j\in (k,\ldots,n)$.
According to the multilinearity of the determinant, we have
$$M_{\alpha}^{\alpha}(D^2 u_m)=  (Q_{\overline{\gamma}-j})^k \sum_{\mathcal{L_{\alpha}}}C(\mathcal{L_{\alpha}}) \det (H(\mathcal{L_{\alpha}},j)),$$
where the sum is over $\mathcal{L_{\alpha}}=(l_1,\ldots,l_{k-1},l_j)\in \{1,\ldots,m\}^k$, we set
$$C(\mathcal{L_{\alpha}})= \prod_{i\in \alpha} \frac{1}{(n_{l_i})^{2-\frac{2}{k}} (l_i)^{\frac{1}{k}}},$$
and the $k\times k$ matrix $H(\mathcal{L_{\alpha}},j)$ is given by
$$H(\mathcal{L_{\alpha}},j)=\left( \partial_{s,t} (x_jP_{l_s,\gamma})\right)_{s,t\in \alpha}.$$
We denote $J_0$ the collection of all multi-indices $\mathcal{L_{\alpha}}=(l,l,\ldots, l)$ for $l\in \{1,2,\ldots,m\}$. Hence
\begin{align*}
\MyRoman{1}&=\left|\sum_{j\in \overline{\gamma}} \int M_{\gamma+j}^{\gamma+j}(D^2 u_m) \varphi dx\right|\\
&=\left|\sum_{j\in \overline{\gamma}} \int (Q_{\overline{\gamma}-j})^k \sum_{\mathcal{L_{\alpha}}}  C(\mathcal{L_{\alpha}}) \det (H(\mathcal{L_{\alpha}},j)) \varphi dx\right|\\
&\geq \MyRoman{3}- \MyRoman{4},
\end{align*}
where
\begin{equation}\label{formulaor36}
\MyRoman{3}:= \left|\sum_{j\in \overline{\gamma}} \int (Q_{\overline{\gamma}-j})^k \sum_{\mathcal{L_{\alpha}}\in J_0}  C(\mathcal{L_{\alpha}}) \det (H(\mathcal{L},j)) \varphi dx\right|,
\end{equation}
and
\begin{equation}\label{formulaor37}
\MyRoman{4}:= \left|\sum_{j\in \overline{\gamma}} \int (Q_{\overline{\gamma}-j})^k \sum_{\mathcal{L_{\alpha}}\notin J_0}  C(\mathcal{L_{\alpha}}) \det (H(\mathcal{L},j)) \varphi dx\right|.
\end{equation}
Since
$$\det (H(\mathcal{L_{\alpha}},j))= (-1)^{k-1} 2^k n_l^{2(k-1)} (\prod_{i\in \gamma} \sin^2 (n_l x_i))^{k-1} (x_j)^{k-2} (\sum_{i\in \gamma} \cos^2(n_lx_i))$$
for any $\mathcal{L_{\alpha}}=(l,l,\cdot\cdot\cdot,l)\in J_0$,  then we have
\begin{align*}
\MyRoman{3}&= \left|\sum_{j\in \overline{\gamma}} \sum_{l=1}^m (-1)^{k-1} 2^k \frac{1}{l} \int (Q_{\overline{\gamma}-j})^k (\prod_{i\in \gamma} \sin^2 (n_l x_i))^{k-1} (x_j)^{k-2} (\sum_{i\in \gamma} \cos^2(n_l x_i)) \varphi(x) dx\right|\\
&=2^k \left|\sum_{l=1}^m  \frac{1}{l} \sum_{j\in \overline{\gamma}} \int_{(0,2\pi)^n} (Q_{\overline{\gamma}-j})^k  (\prod_{i\in \gamma} \sin^2 (n_l x_i))^{k-1} (x_j)^{k-2} (\sum_{i\in \gamma} \cos^2(n_l x_i))  \varphi(x) dx\right|\\
&\geq C_1 \log m,
\end{align*}
where  $C_1$ is a positive constant independent of $m$. Note that
\begin{align*}
\MyRoman{4}&\leq \sum_{j\in \overline{\gamma}}  \left| \sum_{\mathcal{L}_{\alpha}\notin J_0} C(\mathcal{L_{\alpha}}) \int (Q_{\overline{\gamma}-j})^k \det (H(\mathcal{L}_{\alpha},j))\varphi dx\right|\\
&\leq\sum_{j\in \overline{\gamma}}  \left| \sum_{\mathcal{L}_{\alpha}\notin J_0} C(\mathcal{L_{\alpha}}) \int_{\mathbb{R}^k} \det (H(\mathcal{L}_{\alpha},j))\prod_{i\in \alpha}\varphi_i(x_i) dx_{\alpha} \right|
\cdot \left|\int_{\mathbb{R}^{n-k}} (Q_{\overline{\gamma}-j})^k  \prod_{i\in \overline{\alpha}}\varphi_i(x_i) dx_{\overline{\alpha}}\right|\\
&\leq C \|\varphi_{\overline{\alpha}}\|_{L^{\infty}} \sum_{j\in \overline{\gamma}} \sum_{\mathcal{L}_{\alpha}\notin J_0} C(\mathcal{L}_{\alpha})\left|\int_{\mathbb{R}^k} \det (H(\mathcal{L}_{\alpha},j))\prod_{i\in \alpha}\varphi_i(x_i) dx_{\alpha}\right|.
\end{align*}
Then  Proposition 5.2 in \cite{BJ} implies that there exists a constant $C>0$ such that
$$\MyRoman{4}\leq C \sum_{\mathcal{L}_{\alpha}\notin J_0} \frac{\|\varphi\|_{C^2}}{m^{2k}}\leq C\|\varphi\|_{C^2}.$$
Hence
$$\MyRoman{1}\geq C_1 \log m- C \|\varphi\|_{C^2}. $$

 Step 2:  Estimate of $\left| \sum_{\alpha\in I_0} \int M_{\alpha}^{\alpha}(D^2u_m)\varphi dx\right|$.

 Without loss of generality we can assume that $I_0\neq \emptyset$. Then for any $\alpha \in I(k,n)\cap I_0$ we have
 $$M_{\alpha}^{\alpha} ( D^2 u_m)= \sum_{\mathcal{L}_{\alpha}} C(\mathcal{L}_{\alpha}) \prod_{i\in \gamma} \sin^2(n_{l_i} x_i)(Q_{\overline{\gamma}-\alpha})^k \det(G_{\alpha}),$$
 where the $k\times k$ matrix $G_{\alpha}$ is  given by
 $$G_{\alpha}=\left(\partial_{st} \left(\prod_{i\in \alpha} x_i\right)\right)_{s,t\in \alpha}.$$
 Therefore
 \begin{align*}
\left| \sum_{\alpha\in I_0} \int M_{\alpha}^{\alpha}(D^2u_m)\varphi dx\right|&\leq \|\varphi\|_{L^{\infty}} \sum_{\alpha\in I_0} \sum_{\mathcal{L}_{\alpha}} C(\mathcal{L}_{\alpha}) \int_{(0,2\pi)^n}\prod_{i\in \gamma} \sin^2(n_{l_i} x_i)(Q_{\overline{\gamma}-\alpha})^k |\det(G_{\alpha})|dx\\
&\leq C \|\varphi\|_{L^{\infty}}.
\end{align*}

 Step 3: The first estimate of $\left| \sum_{\alpha\in I_c} \int M_{\alpha}^{\alpha}(D^2u_m)\varphi dx\right|$ for $c=1,\ldots,k-2$.

 For any $\alpha \in I_c$, it can be written as $\alpha=(\alpha',\alpha^{''})$ with $\alpha' \subset \gamma$ and $\alpha^{''}\subset \overline{\gamma}$. Set
 $$y_1=x_{\alpha'_1},~y_2=x_{\alpha'_2},~\cdot\cdot\cdot,~y_c=x_{\alpha'_c},~y_{c+1}=x_{^{\alpha''_1}},
 ~y_{c+2}=x_{^{\alpha''_2}},\cdot\cdot\cdot,~y_k=x_{^{\alpha''_{k-c}}}.$$
 Hence
 \begin{equation}\label{formulaor38}
u_{m}= \sum_{l=1}^m \frac{1}{n_l^{2-\frac{2}{k}}l^{\frac{1}{k}}} \left(\prod_{s=1}^c \sin^2(n_ly_s) \prod_{s=1}^{k-c}y_{c+s}\right) \prod_{i\in\gamma-\alpha'} \sin^2(n_l x_i) \prod_{i\in \overline{\gamma}-\alpha^{''}} x_i,
\end{equation}
which implies
\begin{align*}
&\left|  \int M_{\alpha}^{\alpha}(D^2u_m)\varphi dx\right| \\
&=\left|  \int \left(\prod_{i\in\gamma-\alpha'} \sin^2(n_l x_i) \prod_{i\in \overline{\gamma}-\alpha^{''}} x_i\right)^k \det\left(D^2\left(\sum_{l=1}^m \frac{1}{n_l^{2-\frac{2}{k}}l^{\frac{1}{k}}} \left(\prod_{s=1}^c \sin^2(n_ly_s) \prod_{s=1}^{k-c}y_{c+s}\right)\right)\right)\varphi dx_{\overline{\alpha}}dy\right|\\
&\leq  \left| \int_{(0,2\pi)^{n-k}} \left(\prod_{i\in\gamma-\alpha'} \sin^2(n_l x_i) \prod_{i\in \overline{\gamma}-\alpha^{''}} x_i\right)^k  \prod_{i\in \overline{\alpha}}\varphi_i(x_i)dx_{\overline{\alpha}}\right|\\
&\cdot \left| \int_{\mathbb{R}^k}  \det\left(D^2\left(\sum_{l=1}^m \frac{1}{n_l^{2-\frac{2}{k}}l^{\frac{1}{k}}} \left(\prod_{s=1}^c \sin^2(n_ly_s) \prod_{s=1}^{k-c}y_{c+s}\right)\right)\right) \prod_{i\in \alpha} \varphi_i(x_i)dy\right|\\
&\leq C\|\varphi\|_{L^{\infty}}  \left| \int_{\mathbb{R}^k}  \det\left(D^2\left(\sum_{l=1}^m \frac{1}{n_l^{2-\frac{2}{k}}l^{\frac{1}{k}}} \left(\prod_{s=1}^c \sin^2(n_ly_s) \prod_{s=1}^{k-c}y_{c+s}\right)\right)\right) \prod_{i\in \alpha} \varphi_i(x_i)dy\right|,
\end{align*}
where $y=(y_1,\ldots,y_k)$. So it is convenient to set
\begin{equation}\label{formulaor39}
v_m(y)=\sum_{l=1}^m \frac{1}{n_l^{2-\frac{2}{k}}l^{\frac{1}{k}}} \prod_{s=1}^c \sin^2(n_ly_s) \prod_{s=1}^{n-c}y_{c+s},~~\psi(y)=\prod_{i\in \alpha} \varphi_i(x_i).
\end{equation}
In order to estimate  $\left| \sum_{\alpha\in I_c} \int M_{\alpha}^{\alpha}(D^2u_m)\varphi dx\right|$, it is sufficient to show that
\begin{equation}\label{formulaor310}
\MyRoman{5}:= \left| \int_{\mathbb{R}^k}  \det\left(D^2 v_m\right) \psi dy\right| \leq C\|\psi\|_{C^2},
\end{equation}
where $C>0$ is a  constant.

Step 4: The first estimate for $\MyRoman{5}$.

 Similarly,  we define $P_l:\mathbb{R}^c \rightarrow \mathbb{R}$ and $Q:\mathbb{R}^{k-c} \rightarrow \mathbb{R}$ by
$$P_l:=\prod_{i=1}^{c} \sin^2(n_l y_i),~~ Q:=\prod_{i=i}^{k-c} y_{c+i}.$$
Similar to Step 2, we have
$$\MyRoman{5}=\left|\sum_{\mathcal{L}} C(\mathcal{L}) \int det(H_{\mathcal{L}}) \psi dy\right|,$$
where $\mathcal{L}=(l_1,\ldots,l_k)\in \{1,2,\ldots,m\}^k$,
$$C(\mathcal{L})=\prod_{i=1}^k \frac{1}{n_{l_i}^{2-\frac{2}{k}} (l_i)^{\frac{1}{k}}},$$
and  the $k\times k$ matrix $H_{\mathcal{L}}=H_{\mathcal{L}}(y)$ is given by
$$\left( \partial_{ij} (P_{l_i}Q)\right)_{i,j \in (1,2,\cdot\cdot\cdot,k)}.$$
Fixing $\mathcal{L}=(l_1,\ldots,l_k)$, denote
$$l_{\ast}:=\max\{l_i\mid i=1,\ldots,k\},$$
and define
$$\beta_{\mathcal{L}}:=\{i:l_i=l_{\ast}\}.$$
Using Laplace formulas of the determinant we obtain
\begin{align*}
\MyRoman{5}&=\left|\sum_{\rho=1}^k \sum_{|\beta_{\mathcal{L}}|=\rho} C(\mathcal{L}) \int det(H_{\mathcal{L}}) \psi dy\right|\\
&=\left| \sum_{\rho=1}^k \sum_{|\beta_{\mathcal{L}}|=\rho}  C(\mathcal{L}) \sum_{\xi\in I(\rho,k)} \sigma(\beta_{\mathcal{L}}, \overline{ \beta_{ \mathcal{L} } }) \sigma (\xi,\overline{\xi})\int M_{\beta_{\mathcal{L}}}^{\xi}(H_{\mathcal{L}}) M_{\overline{\beta_{\mathcal{L}}}}^{\overline{\xi}}  (H_{\mathcal{L}})
 \psi dy\right|\\
 &\leq \MyRoman{6}+\MyRoman{7},
\end{align*}
where
\begin{equation}\label{formulaor310}
\MyRoman{6}:= \sum_{\rho=1}^c \sum_{|\beta_{\mathcal{L}}|=\rho}  \sum_{\xi\in I(\rho,k)} \left|C(\mathcal{L})\int M_{\beta_{\mathcal{L}}}^{\xi}(H_{\mathcal{L}}) M_{\overline{\beta_{\mathcal{L}}}}^{\overline{\xi}}  (H_{\mathcal{L}})
 \psi dy\right|,
\end{equation}
and
\begin{equation}\label{formulaor311}
\MyRoman{7}:= \sum_{\rho=c+1}^k \sum_{|\beta_{\mathcal{L}}|=\rho}  \sum_{\xi\in I(\rho,k)} \left|C(\mathcal{L})\int M_{\beta_{\mathcal{L}}}^{\xi}(H_{\mathcal{L}}) M_{\overline{\beta_{\mathcal{L}}}}^{\overline{\xi}}  (H_{\mathcal{L}})
 \psi dy\right|.
\end{equation}
Note that we separated the determinant $\det(H_{\mathcal{L}})$ into two parts: $M_{\beta_{\mathcal{L}}}^{\xi}(H_{\mathcal{L}})$ involves only frequencies of the highest order $n_{l_{\ast}}$ or $0$, while $M_{\overline{\beta_{\mathcal{L}}}}^{\overline{\xi}}  (H_{\mathcal{L}})$ involves only frequencies of lower order $n_{l_i}$ with $l_i\leq n_{l_{\ast}-1}$.

If $\rho>c$, for any  $\mathcal{L}$ with $|\beta_{\mathcal{L}}|=\rho$ and $\xi \in I(\rho,k)$,   we set $|\beta_{\mathcal{L}}\cap(1,2,\ldots,c)|=b$ and $|\xi\cap (1,2,\ldots,c)|=b'$. There is no loss of generality in assuming $\beta_{\mathcal{L}}=(1,2,\ldots,b,\beta_{b+1},\ldots,\beta_{\rho})$, $\xi=(1,2,\ldots,b',\beta_{b'+1},\ldots,\beta_{\rho})$. Then
$$(H_{\mathcal{L}})_{\beta_{\mathcal{L}}}^{\xi}=\left(
  \begin{array}{cc}
    n_{l_{\ast}}^2 g_{1,1},  \cdots,  n_{l_{\ast}}^2 g_{1,b'}, n_{l_{\ast}} g_{1,b'+1},\cdots, n_{l_{\ast}} g_{1,a}\\
    \cdot\cdot\cdot  \\
    n_{l_{\ast}}^2 g_{b,1}, \cdots,  n_{l_{\ast}}^2 g_{b,b'}, n_{l_{\ast}} g_{b,b'+1},\cdots, n_{l_{\ast}} g_{b,a}\\
    n_{l_{\ast}} g_{b+1,1}, \cdots,  n_{l_{\ast}} g_{b+1,b'},  g_{b+1,b'+1},\cdots,  g_{b+1,a}\\
    \cdot\cdot\cdot\\
    n_{l_{\ast}} g_{a,1}, \cdots,  n_{l_{\ast}} g_{a,b'},  g_{a,b'+1},\cdots,  g_{a,a}\\
  \end{array}
\right)$$
where $g_{s,t}$ is a uniformly bounded function for  $s\in \beta_{\mathcal{L}}$, $t\in \xi$. It follows that
$$|M_{\beta_{\mathcal{L}}}^{\xi}(H_{\mathcal{L}}) |\leq C n_{l_{\ast}}^{b+b'}\leq C n_{l_{\ast}}^{2c}.$$
The following result may be proved in much the same way as above:
$$|M_{\overline{\beta_{\mathcal{L}}}}^{\overline{\xi}}  (H_{\mathcal{L}})| \leq C n_{l_{\ast}-1}^{2(k-\rho)}.$$
Hence
\begin{align*}
\MyRoman{7} &\leq C \sum_{\rho=c+1}^k \sum_{|\beta_{\mathcal{L}}|=\rho}  \sum_{\xi\in I(\rho,k)}   \frac{1}{n_{l_{\ast}}^{(2-\frac{2}{k})\rho}} \int_{(0,2\pi)^k} |M_{\beta_{\mathcal{L}}}^{\xi}(H_{\mathcal{L}}) | |M_{\overline{\beta_{\mathcal{L}}}}^{\overline{\xi}}  (H_{\mathcal{L}})| |\psi| dy \\
&\leq C \|\psi\|_{L^{\infty}} \sum_{\rho=c+1}^k \frac{n_{l_{\ast}-1}^{2(k-\rho)}}{n_{l_{\ast}}^{2\rho-\frac{2\rho}{k}-2c}}\\
&\leq C \|\psi\|_{L^{\infty}}  \sum_{\rho=c+1}^{k-1} \frac{n_{l_{\ast}-1}^{2k}}{n_{l_{\ast}}^{\frac{2}{k}}} \\
&\leq C \frac{\|\psi\|_{L^{\infty}}}{m^{2k}}.
\end{align*}
Obviously we shall have established the theorem if we could estimate $\MyRoman{6}$.

Step 5: Fix $\mathcal{L}$ such that $|\beta_{\mathcal{L}}|=\rho\leq c$, and we will prove that
\begin{equation}\label{formulaor312}
M_{\beta_{\mathcal{L}}}^{\xi}(H_{\mathcal{L}})=0
\end{equation}
for any $\xi \in I(\rho,k)$ with $|\beta_{\mathcal{L}}\cap\xi|\leq \rho-2$.

Let $i_1,i_2 \in \beta_{\mathcal{L}} \backslash\xi$ be given with $i_1\neq i_2$, and set
$$h_i(y_i):=\begin{cases}
\sin^2(n_{l_{\ast}} y_i), &i\in 1,\ldots,c \\
y_i,  &i\in c+1,\ldots,k
\end{cases}$$
and $H(y)=\prod_{i=1}^kh_i(y_i)$,
$$v_k(y)=\left(\partial_{i_k,j} H\right)_{j\in \xi}\in \mathbb{R}^\rho,~~~~k=1,2.$$
Since $i_1,i_2 \notin \xi$, we have
$$\begin{cases}
v_1=\left(h'_{i_1} h'_{j} H_{\overline{i_1+j}}\right)_{j\in \xi},\\
v_2=\left(h'_{i_2} h'_{j} H_{\overline{i_2+j}}\right)_{j\in \xi},
\end{cases}
$$
which immediately give (\ref{formulaor312}).

Let $\rho\leq c$, $\beta_{\mathcal{L}}$ and $\xi$ be given. If either
\begin{enumerate}
\item[(\romannumeral1)]  $|\beta_{\mathcal{L}}\cap \xi|=\rho$ such that $\beta_{\mathcal{L}} \cap (c+1,c+2,\ldots,k)\neq\emptyset$, or
\item [(\romannumeral2)]  $|\beta_{\mathcal{L}}\cap \xi|=\rho-1$ such that $i^{\ast}:=\beta_{\mathcal{L}}\backslash \xi$, $j^{\ast}:= \xi \backslash\beta_{\mathcal{L}} \in (c+1,\ldots,k)$,
\end{enumerate}
is satisfied, by the same method as in Step 4 and (5.24) in \cite{BJ}, it follows that
$$\left|C_{\mathcal{L}}\int M_{\beta_{\mathcal{L}}}^{\xi}(H_{\mathcal{L}}) M_{\overline{\beta_{\mathcal{L}}}}^{\overline{\xi}}  (H_{\mathcal{L}})
 \psi dy\right| \leq C \frac{\|\psi\|_{L^{\infty}}}{m^{2k}}.$$
Set
\begin{align*}
S_{\rho}:&=\{(\mathcal{L},\xi)\mid \beta_{\mathcal{L}},\xi\in I(\rho,k), |\beta_{\mathcal{L}}\cap\xi|=\rho, \beta_{\mathcal{L}} \cap (c+1,\cdot\cdot\cdot,k)=\emptyset\}\\
&\cup\{(\mathcal{L},\xi)\mid \beta_{\mathcal{L}},\xi\in I(\rho,k), |\beta_{\mathcal{L}}\cap\xi|=\rho-1, \beta_{\mathcal{L}} \backslash\xi \notin (c+1,\cdot\cdot\cdot,k)\}\\
 &\cup\{(\mathcal{L},\xi)\mid \beta_{\mathcal{L}},\xi\in I(\rho,k), |\beta_{\mathcal{L}}\cap\xi|=\rho-1,\xi \backslash \beta_{\mathcal{L}} \notin (c+1,\cdot\cdot\cdot,k)\},
\end{align*}
then
\begin{equation}\label{formulaor313}
\MyRoman{6}\leq  C \frac{\|\psi\|_{L^{\infty}}}{m^{2k}}+\sum_{\rho=1}^c \sum_{(\mathcal{L},\xi)\in S_{\rho}}  \left|C(\mathcal{L})\int M_{\beta_{\mathcal{L}}}^{\xi}(H_{\mathcal{L}}) M_{\overline{\beta_{\mathcal{L}}}}^{\overline{\xi}}  (H_{\mathcal{L}})
 \psi dy\right|.
\end{equation}
It is easy to see that for any $\xi\in I(\rho,k)$ there exist integers $b_{c+1},b_{c+2},\ldots,b_k\leq \rho$ and a sequence of coefficients $\{c_z\}\subset \mathbb{C}$ such that
\begin{equation}\label{formulaor314}
M_{\beta_{\mathcal{L}}}^{\xi}(H_{\mathcal{L}})=\sum_{z\in \Lambda} c_z e^{2n_{l_{\ast}}iz\cdot\widehat{y}}y_{c+1}^{b_{c+1}}y_{c+2}^{b_{c+2}}\cdot\cdot\cdot y_k^{b_k},
\end{equation}
where $\widehat{y}=(y_1,y_2,\cdot\cdot\cdot,y_c)$,
$$\Lambda=\{z\in \mathbb{Z}^c\mid |z_i|\leq c\},$$
and
$$|c_z|\leq Cn_{l_{\ast}}^{2\rho}.$$
In fact the proof of this statement follows in a similar manner in \cite[Remark 5.5]{BJ}.

Step 6: Next we have to show that $c_{(0,\ldots,0)}=0$ for any $(\mathcal{L},\xi) \in S_{\rho}$ where $c_{(0,\ldots,0)}$ is defined in (\ref{formulaor314}).

According to (\ref{formulaor314}), it suffices to show that
\begin{equation}\label{formulaor315}
\int_{[0,2\pi]^c} M_{\beta_{\mathcal{L}}}^{\xi}(H_{\mathcal{L}}) d\widehat{y}=0
\end{equation}
for  each $(\mathcal{L},\xi)\in S_{\rho}$.
Suppose that $\beta_{\mathcal{L}}=\xi$ and  $\beta_{\mathcal{L}} \cap (c+1,\ldots,k)=\emptyset$, and set $\eta:=(1,\cdot\cdot\cdot,c)$. Then
\begin{align*}
\int_{[0,2\pi]^c} M_{\beta_{\mathcal{L}}}^{\xi}(H_{\mathcal{L}}) d\widehat{y}&=\int_{[0,2\pi]^c} (P_{l_{\ast},\eta-\xi}Q)^{\rho} \det \left(D^2 \left(\prod_{i\in\xi} \sin^2(n_{l_{\ast}} y_i)\right)\right) d\widehat{y}\\
&=Q^{\rho}\int_{[0,2\pi]^{c-\rho}} (P_{l_{\ast},\eta-\xi})^{\rho} dy_{\eta-\xi} \\
&\cdot (-2n^2_{l_{\ast}})^{\rho}\int_{[0,2\pi]^{\rho}} \left(\prod_{i\in\xi} \sin(n_{l_{\ast}} y_i)\right)^{2\rho-2} \left(1-2\sum_{i\in\xi} \cos^2(n_{l_{\ast}} y_i)\right) dy_{\xi}.
\end{align*}
The equality (\ref{formulaor315}) holds as desired due to the equality (5.34) in \cite{BJ}.

We now turn to the second case, suppose that $|\beta_{\mathcal{L}}\cap\xi|=\rho-1$ and $j^{\ast}= \xi \backslash \beta_{\mathcal{L}}\in (1,2,\ldots,c)$.
Using the Laplace formulas of determinant again we obtain
\begin{align*}
M_{\beta_{\mathcal{L}}}^{\xi}(H_{\mathcal{L}})&= \sum_{i\in \beta_{\mathcal{L}}} \sigma(i,\beta_{\mathcal{L}}-i) \sigma(j^{\ast},\xi-j^{\ast}) \partial_{i,j^{\ast}} (P_{l_{\ast}}Q) M_{\beta_{\mathcal{L}}-i}^{\xi-j^{\ast}}(H_{\mathcal{L}})\\
&=n_{l_{\ast}} \sin(2n_{l_{\ast}}y_{j^{\ast}}) \sin^{2\rho-2}(n_{l_{\ast}}y_{j^{\ast}}) g(y),
\end{align*}
where the function $g: \mathbb{R}^k \rightarrow \mathbb{R}$ is independent of the variable $y_{j^{\ast}}$.
It follows that $M_{\beta_{\mathcal{L}}}^{\xi}(H_{\mathcal{L}})$ is an odd function in the variable $y_{j^{\ast}}$,
so the equality (\ref{formulaor315}) is obtained.

The proof of the last case for this statement follows in a similar manner which implies  $c_{(0,\ldots,0)}=0$ for any $(\mathcal{L},\xi) \in S_{\rho}$.

Step 7: Finally we have to estimate the second part  on the right-hand side of (\ref{formulaor313}) by integration by parts.

For any $(\mathcal{L},\xi) \in S_{\rho}$ we have $c_{(0,\ldots,0)}=0$, and then
\begin{align*}
\MyRoman{8}&:=\sum_{\rho=1}^c \sum_{(\mathcal{L},\xi)\in S_{\rho}}  \left|C(\mathcal{L})\int M_{\beta_{\mathcal{L}}}^{\xi}(H_{\mathcal{L}}) M_{\overline{\beta_{\mathcal{L}}}}^{\overline{\xi}}  (H_{\mathcal{L}}) \psi dy\right|\\
 &\leq C \sum_{\rho=1}^c (n_{l_{\ast}})^{ \frac{2\rho}{k}}\left( \sup_{(\mathcal{L},\xi)\in S, z\in \Lambda \backslash \{0\}}\left| \int e^{2n_{l_{\ast}}iz\cdot\widehat{y}}y_{c+1}^{b_{c+1}}y_{c+2}^{b_{c+2}}\cdots y_k^{b_k} M_{\overline{\beta_{\mathcal{L}}}}^{\overline{\xi}}  (H_{\mathcal{L}})
 \psi dy\right|\right)
\end{align*}
where $C>0$ is a constant. Let $z=(z_1,\ldots,z_c)\in \Lambda \backslash \{0\}$ be given, there exists $j\in (1,\ldots,c)$ such that $z_j\neq 0$. Using the integration by parts two times in the $y_j$ variable, we obtain
\begin{equation}\label{formulaor316}
\int e^{2n_{l_{\ast}}iz\cdot\widehat{y}}y_{c+1}^{b_{c+1}}y_{c+2}^{b_{c+2}}\cdots y_k^{b_k} M_{\overline{\beta_{\mathcal{L}}}}^{\overline{\xi}}  (H_{\mathcal{L}})
 \psi dy=-\frac{1}{4(n_{l_{\ast}})^2 (z_j)^2} \int e^{2n_{l_{\ast}}iz\cdot\widehat{y}}y_{c+1}^{b_{c+1}}y_{c+2}^{b_{c+2}}\cdots y_k^{b_k} \partial _j^2\left(M_{\overline{\beta_{\mathcal{L}}}}^{\overline{\xi}}  (H_{\mathcal{L}})
 \psi \right) dy.
\end{equation}
In fact
$$\left|\partial^2_j M_{\overline{\beta_{\mathcal{L}}}}^{\overline{\xi}}  (H_{\mathcal{L}}) \right|\leq C (n_{l_{\ast}-1})^{2(k-\rho)+2}\leq  C (n_{l_{\ast}-1})^{2k}.$$
It follows that
\begin{equation}
\MyRoman{8} \leq C \|\psi\|_{C^2} \sum_{\rho=1}^c (n_{l_{\ast}})^{\frac{2\rho}{k}} \frac{(n_{l_{\ast}-1})^{2k}}{(n_{l_{\ast}})^2}
=C \|\psi\|_{C^2} \sum_{\rho=1}^c \frac{(n_{l_{\ast}-1})^{2k}}{(n_{l_{\ast}})^{2-\frac{2\rho}{k}}}
\leq C \frac{\|\psi\|_{C^2}}{m^{2k}}.
\end{equation}
Therefore we establish the estimate of \MyRoman{1}-\MyRoman{8}, which gives the desired result.
\end{proof}

\begin{proof}[\bf The proof of Theorem \ref{thm51}.]
Let
$$\widetilde{u}_m:=\frac{u_m}{(\log m)^{\frac{1}{2k}}},$$
where $u_m$ is defined in (\ref{formulaor32}). Then by Lemma 5.2 and 5.3, Theorem \ref{thm51} is established and hence Theorem 1.2 is completely proved.
\end{proof}

%---------------------------------------------------------------------------------------%

\section*{Acknowledgments}
\addcontentsline{toc}{chapter}{Acknowledgements}
This work is supported by NSF grant of China ( No.11131005, No.11301400) and the Fundamental Research Funds for the Central Universities(Grant No. 2014201020203).

\bibliographystyle{plain}

\end{document}